\documentclass[10pt]{amsart}

\usepackage[T1]{fontenc}
\usepackage[english]{babel}

\usepackage[a4paper,margin=3cm]{geometry}
\allowdisplaybreaks
\usepackage{amssymb}

\usepackage[
hyperfootnotes=false
]{hyperref}

\hypersetup{
    colorlinks=true,
    linkcolor=blue,
    citecolor=blue,
    urlcolor=red
}

\theoremstyle{plain}
\newtheorem{lemma}{Lemma}[section]
\newtheorem{theorem}[lemma]{Theorem}
\newtheorem{proposition}[lemma]{Proposition}

\theoremstyle{definition}

\theoremstyle{remark}

\numberwithin{equation}{section}

\newcommand{\R}{\mathbb R}
\newcommand{\C}{\mathbb C}

\newcommand{\supp}{\operatorname{supp}}

\newcommand{\eps}{\varepsilon}

\begin{document}

    \title{Strong Unique Continuation for Fractional Schr\"odinger Operators}

    \begin{abstract}
        We prove the strong unique continuation property for the fractional Schr\"odinger equation
\[
    (-\Delta)^s u=Vu
\]
with scaling-critical potentials
   $ V\in L^{\frac{n}{2s}}_{\mathrm{loc}}$ for all $0<s<1$.
Our result constitutes the nonlocal counterpart to the classical Jerison--Kenig result for Schr\"odinger operators.
    \end{abstract}



    \author[H.~Prasad]{Harsh Prasad}
    \address[H.~Prasad]{Fakult\"at f\"ur Mathematik, Universit\"at Bielefeld, 33615 Bielefeld, Germany.}
    \email[]{hprasad@math.uni-bielefeld.de}

    \maketitle
\setcounter{tocdepth}{1}
\tableofcontents
\section{Introduction}

Unique continuation is a fundamental rigidity principle for elliptic
equations. In its strong form, it asserts that a nontrivial solution cannot
vanish to infinite order at a point. For the Schr\"odinger equation
\begin{equation}\label{eq:intro-local}
-\Delta u=Vu,
\end{equation}
the natural potential space is dictated by scaling. Under the rescaling
\(u_\lambda(x)=u(\lambda x)\), the potential transforms as
\[
    V_\lambda(x)=\lambda^2V(\lambda x),
\]

and the corresponding scale-invariant Lebesgue space is \(L^{n/2}\). The
landmark theorem of Jerison and Kenig \cite{JerisonKenig1985}, established strong unique continuation at
this critical scale. The purpose of the present paper is to prove the
corresponding nonlocal theorem for the fractional Laplacian.

Consider the fractional Schr\"odinger equation
\begin{equation}\label{eq:intro-frac}
(-\Delta)^s u=Vu,
\qquad 0<s<1.
\end{equation}
The potential now scales according to
\[
    V_\lambda(x)=\lambda^{2s}V(\lambda x),
\]

so the scale-invariant Lebesgue space is
\begin{equation}\label{eq:critical-space-intro}
L^{\frac{n}{2s}}.
\end{equation}
The precise nonlocal analogue of the Jerison--Kenig problem is therefore
whether \eqref{eq:intro-frac} has the strong unique continuation property for
an arbitrary potential in \eqref{eq:critical-space-intro}.

Unique continuation for fractional equations has been developed principally
through the extension theorem of Caffarelli and Silvestre
\cite{CaffarelliSilvestre2007}. Fall and Felli \cite{FallFelli2014} used
frequency-function and blow-up methods to obtain local asymptotics and strong
unique continuation for equations with Hardy-type potentials. R\"uland
\cite{Ruland2015} developed extension-based Carleman estimates for rough and
scaling-critical potentials, while Yu \cite{Yu2017} treated fractional
powers of variable-coefficient elliptic operators. Higher-order fractional
equations were subsequently considered in
\cite{FelliFerrero2020,GarciaFerreroRuland2019}.

The strong unique continuation results obtained by these methods impose
structural, pointwise, or regularity assumptions on the potential. For
example, R\"uland's principal result treats potentials of the form

$$
    V(x)
    =
    |x|^{-2s}h\!\left(\frac{x}{|x|}\right)+V_2(x),
    \qquad
    |V_2(x)|\lesssim |x|^{-2s+\varepsilon},
$$

with additional hypotheses in the low-order regime; see
\cite[Proposition~1.1]{Ruland2015}. In particular, for arbitrary rough
potentials without differentiability, the extension-Carleman argument is
restricted to \(s\ge\frac14\).

A different, direct approach was developed by Seo
\cite{Seo2014MathNachr,Seo2015Proc,Seo2015Taiwan}. In
\cite{Seo2015Proc}, Seo proved weak unique continuation at the critical
Lorentz scale
\[
    V\in L^{n/(2s),\infty}_{\mathrm{loc}},
\]

under a local smallness condition. More recently, the author
\cite{Prasad2026} obtained global weak unique continuation for substantially
rougher potentials in \(H^{-s}_{\mathrm{loc}}\).  These results show that weak unique
continuation persists well beyond the scale-critical Lebesgue space. On the other hand, Grube
\cite{Grube2026} recently proved a conditional strong-type unique
continuation principle. His
result assumes additional regularity \cite[Remark~1.2 (i)]{Grube2026} and a moment condition stronger than
flatness \cite[Remark~1.4]{Grube2026}.

Therefore, despite substantial progress over the past decade, strong unique continuation for scaling critical potentials has remained open. Indeed, as our discussion above illustrates, even the subcritical case of bounded, measurable potentials previously required restrictions on $s$. We now state our main

\begin{theorem}\label{thm:intro-main}
Let \(n\ge3\), \(0<s<1\), and let \(\Omega\subset\R^n\) be open. Suppose

$$
    V\in L^{\frac{n}{2s}}_{\mathrm{loc}}(\Omega;\C),
    \qquad
    u\in H^s(\R^n;\C),
$$

and

$$
    (-\Delta)^s u=Vu
    \qquad\text{in }\Omega
$$

in the sense of distributions. If, for some \(x_0\in\Omega\),
\begin{equation}\label{eq:flatness}
\int_{B_r(x_0)}|u|^2,dx=O(r^N)
\qquad\text{as }r\downarrow0
\end{equation}
for every \(N>0\), then \(u\equiv0\) in \(\R^n\).
\end{theorem}

Theorem~\ref{thm:intro-main} is the nonlocal counterpart of the
Jerison--Kenig theorem. The proof is direct and does not use the extension problem. Its main new
ingredient is a cancellation lemma which overcomes the obstruction created
by nonlocal localization and permits the usage of Carleman estimates as in \cite{JerisonKenig1985}. The
strategy is explained in Section~\ref{sec:proof-sketch}.

\medskip

\noindent\textbf{Outline.}
Section~\ref{sec:proof-sketch} explains the main ideas of the proof.
Section~\ref{sec:prelim} collects the required properties of the fractional
Laplacian and the endpoint bounds from
\cite{JerisonKenig1985,Stein1985}. Section~\ref{sec:aux} is concerned with
some Carleman estimates. Section~\ref{sec:cancellation} uses
flatness and  Carleman estimates to establish the cancellation lemma.
Section~\ref{sec:mainproof} combines these ingredients to prove
Theorem~\ref{thm:intro-main}.

\medskip

\noindent\textbf{LLM Declaration.}
Large language models were used in the preparation of this manuscript. The
author takes responsibility for its contents.

\section{Main Ideas}\label{sec:proof-sketch}

For the Schr\"odinger equation
\[
    -\Delta u=Vu,
\]
the Jerison--Kenig proof relies on a Carleman estimate of the form
\[
 \bigl\||x|^{-t}w\bigr\|_{L^p(dx/|x|^n)}
 \le C_t\bigl\||x|^{2-t}\Delta w\bigr\|_{L^q(dx/|x|^n)},
 \qquad
 \frac1q-\frac1p=\frac2n,
\]
for large noninteger \(t\). Applying it to a cutoff of \(u\) gives a term
involving \(Vu\) and the commutator \([-\Delta,\chi]u\). The latter is
supported where the cutoff changes. Hence, on a sufficiently small ball,
H\"older's inequality and the smallness of \(\|V\|_{L^{n/2}}\) absorb the
potential term, while the weighted commutator term tends to zero as
\(t\to\infty\).

For \((-\Delta)^s\), the formal analogue would replace \(2\) by \(2s\) and
use
\[
    \frac1Q-\frac1P=\frac{2s}{n},
\]
which is the relation paired with \(V\in L^{n/(2s)}\). The difficulty is the
cutoff commutator:
\[
 [(-\Delta)^s,\chi]u(x)
 =c_{n,s}\int_{\mathbb R^n}
   \frac{\chi(x)-\chi(y)}{|x-y|^{n+2s}}u(y)\,dy.
\]
Even when \(x\) lies in the region where \(\chi\equiv1\), this expression
sees all exterior values of \(u\). It is therefore not supported in an
annulus and is generally nonzero near the vanishing point. A large singular
weight can consequently make its contribution non-negligible, or even
non-integrable.

A natural alternative is to leave the equation unchanged near the vanishing
point and localize only its right-hand side. Let \(\chi\) equal one near the
origin and set
\[
    F:=\chi Vu.
\]
Then \(F=Vu\) near the origin, while \(F\) is compactly supported. With
\(\alpha:=2s\), the Riesz potential \(I_\alpha F\) satisfies
\[
    (-\Delta)^sI_\alpha F=F.
\]
Consequently,
\[
    h:=u-I_\alpha F
\]
is \(s\)-harmonic near the origin.

One may now try to prove a Carleman estimate of the form
\[
    \||x|^{-a_m}I_\alpha F\|_{L^P}
    \le C\||x|^{-a_m}F\|_{L^Q},
    \qquad
    a_m=m-\frac12+\frac nP.
\]
However, examining the Riesz kernel near the origin shows why this estimate cannot
hold. Its first \(m\) terms produce a polynomial \(P_{m-1}^F\) of degree at
most \(m-1\), and, unless this polynomial vanishes,
\[
    |x|^{-a_m}P_{m-1}^F\notin L^P(B_1).
\]
The polynomial must therefore be separated from the part to which the
Carleman estimate is applied. We write
\[
    I_\alpha F=P_{m-1}^F+R_mF
\]
and prove instead
\[
    \||x|^{-a_m}R_mF\|_{L^P}
    \le C\||x|^{-a_m}F\|_{L^Q},
\]
with \(C\) independent of \(m\).

The polynomial \(P_{m-1}^F\) has been removed from the Carleman estimate,
but it remains in the decomposition of \(u\). Moreover, since \(h\) is
analytic near the origin,
\[
    h=T_{m-1}h+h_m,
\]
where \(T_{m-1}h\) is its Taylor polynomial of degree at most \(m-1\).
Therefore
\[
    u=
    \bigl(P_{m-1}^F+T_{m-1}h\bigr)+R_mF+h_m.
\]
A Carleman estimate shows that \(R_mF\) is of strictly
higher order than every polynomial of degree at most \(m-1\); analyticity
gives the same conclusion for \(h_m\). Since \(u\) vanishes to infinite
order, the polynomial in the preceding decomposition must vanish:
\[
    P_{m-1}^F+T_{m-1}h\equiv0.
\]
This is the cancellation lemma, and it gives
\[
    u=R_mF+h_m
\]
for every \(m\ge1\).

Next, we use the same Carleman estimate a second time, together with the
standard absorption argument to conclude that \(u\) vanishes on a ball; antilocality of the fractional Laplacian
then forces \(u\equiv0\) in \(\mathbb R^n\).

\section{Preliminaries}\label{sec:prelim}

We first fix the exponents and Riesz-potential notation used throughout. For the rest of the paper we set
\[
    \alpha:=2s\in(0,2),
    \qquad
    p_*:=\frac{2n}{n-\alpha}.
\]
We also fix once and for all
\begin{equation}\label{eq:fixed-parameters}
    \beta:=n-\frac34,
    \qquad
    \eps:=\frac{\alpha}{8n(n-1)},
\end{equation}
and define $P,Q$ by
\begin{equation}\label{eq:PQ}
    \frac1Q=\frac12+\frac{\alpha}{2n}+\eps,
    \qquad
    \frac1P=\frac12-\frac{\alpha}{2n}+\eps.
\end{equation}
For later interpolation we also define $S,R$ by
\begin{equation}\label{eq:SR}
    \frac1S=\frac12+\frac{\beta}{2n}+\frac{\beta}{8n(n-1)},
    \qquad
    \frac1R=\frac12-\frac{\beta}{2n}+\frac{\beta}{8n(n-1)}.
\end{equation}
Notice that $S$ and $R$ depend only on the dimension. Finally, for $m\ge1$ set
\begin{equation}\label{eq:tm-am}
    t_m:=m-\frac12,
    \qquad
    a_m:=t_m+\frac nP=m-\frac12+\frac nP.
\end{equation}

\begin{lemma}\label{lem:exponents}
The exponents above satisfy
\begin{equation}\label{eq:critical-relations}
    1<Q<2<P<p_*,
    \qquad
    \frac1Q-\frac1P=\frac{\alpha}{n},
    \qquad
    Q<\frac n\alpha.
\end{equation}
Moreover
\begin{equation}\label{eq:critical-holder}
    \frac1Q=\frac1{n/\alpha}+\frac1P,
\end{equation}
and
\begin{equation}\label{eq:SR-relations}
    1<S<\frac n\beta,
    \qquad
    \frac1S-\frac1R=\frac\beta n.
\end{equation}
\end{lemma}

\begin{proof}
All identities follow directly from \eqref{eq:PQ} and \eqref{eq:SR}. Since $\eps>0$, one has $P<p_*$ and $Q<2$. Moreover,
\[
    \frac{\alpha}{2n}+\eps<\frac12,
\]
so $1/Q<1$ and hence $Q>1$. Since
\[
    \eps=\frac{\alpha}{8n(n-1)}<\frac{\alpha}{2n},
\]
we have $P>2$. Also $1/Q>\alpha/n$ because $\alpha<n$, so $Q<n/\alpha$. For $S$, the inequality $S<n/\beta$ is equivalent to
\[
    \frac1S>\frac\beta n,
\]
which follows from $\beta<n$. The inequality $S>1$ is equivalent to
\[
    \frac{\beta}{8n(n-1)}<\frac{n-\beta}{2n}=\frac{3}{8n},
\]
which holds for $n\ge3$.
\end{proof}

For \(0<\alpha<n\), let \(I_\alpha\) denote the Riesz potential normalized by
\(\widehat{I_\alpha F}(\xi)=|\xi|^{-\alpha}\widehat F(\xi)\).  We shall use
the Hardy--Littlewood--Sobolev mapping
\begin{equation}\label{eq:HLS}
    I_\alpha:L^Q(\mathbb R^n)\longrightarrow L^P(\mathbb R^n),
    \qquad
    \frac1P=\frac1Q-\frac\alpha n,
    \qquad
    1<Q<\frac n\alpha,
\end{equation}
(see \cite[Chapter~V, \S1]{Stein1970}) and the identity
\begin{equation}\label{eq:Riesz-inverse}
    (-\Delta)^{\alpha/2}I_\alpha F=F \quad \text{in} \quad \mathcal S'(\mathbb R^n),
\end{equation}
which follows from \(\widehat{(-\Delta)^{\alpha/2}I_\alpha F}(\xi)=|\xi|^\alpha|\xi|^{-\alpha}\widehat F(\xi)=\widehat F(\xi)\)
on the Fourier side, extended from Schwartz \(F\) to \(F\in L^Q(\mathbb R^n)\)
by density using \eqref{eq:HLS} and the continuity of \((-\Delta)^{\alpha/2}\)
as a tempered-distribution operator (see also \cite[Chapter~V, \S1]{Stein1970}).

\subsection{Endpoint bounds}
We recall certain endpoint bounds from \cite{JerisonKenig1985, Stein1985}.
We set
\[
    d\mu(x):=\frac{dx}{|x|^n}.
\]
For $m\ge1$ and $0<\Re z<n$, let
\begin{equation}\label{eq:Tz-kernel}
 T_z^{(m)}g(x):=\int_{\mathbb R^n}K_z^{(m)}(x,y)g(y)\,d\mu(y),
\end{equation}
where
\begin{align}
 K_z^{(m)}(x,y)
 &:=
 \frac{c_{n,z}}{\Gamma((n-z)/2)}
 |x|^{-t_m}|y|^{n+t_m-z}
 \left[
 |x-y|^{-n+z}
 -\sum_{j=0}^{m-1}\frac1{j!}
 \left.\partial_r^j|rx-y|^{-n+z}\right|_{r=0}
 \right]                                      \label{eq:Kz-cited},\\
 c_{n,z}
 &:=
 \pi^{-n/2}2^{-z}
 \frac{\Gamma((n-z)/2)}{\Gamma(z/2)}.
\end{align}
The family \eqref{eq:Tz-kernel}--\eqref{eq:Kz-cited} is the normalised
homogeneous-distribution family of \cite[Section~2]{JerisonKenig1985} after choosing $t=t_m$.
For fixed $g\in C_c^\infty(\mathbb R^n\setminus\{0\})$ and $x\ne0$, the map
$z\mapsto T_z^{(m)}g(x)$ extends to an entire, hence continuous, function of
$z$ on all of $\mathbb C$ (\cite[Section~2]{JerisonKenig1985}, where this
entirety is established via the Gamma-function normalisation
$c_{n,z}/\Gamma((n-z)/2)$). The boundary operators
$T_{i\gamma}^{(m)}$ and $T_{\beta+i\gamma}^{(m)}$ appearing below are
therefore unambiguously defined as $T_{i\gamma}^{(m)}g(x)=\lim_{\varepsilon\downarrow0}T_{\varepsilon+i\gamma}^{(m)}g(x)$
and likewise at $\Re z=\beta$, exactly as in
\cite[Section~3]{JerisonKenig1985}.

\begin{lemma}\label{lem:classical-endpoints}
For every \(m\ge1\), \(g\in C_c^\infty(\mathbb R^n\setminus\{0\})\), and
\(\gamma\in\mathbb R\),
\begin{align}
 \|T_{i\gamma}^{(m)}g\|_{L^2(d\mu)}
 &\le C_ne^{c_n|\gamma|}\|g\|_{L^2(d\mu)},                    \label{eq:JK-endpoint}\\
 \|T_{\beta+i\gamma}^{(m)}g\|_{L^R(d\mu)}
 &\le C_ne^{c_n|\gamma|}\|g\|_{L^S(d\mu)}.                   \label{eq:Stein-endpoint}
\end{align}
The constants are independent of \(m\).
\end{lemma}

\begin{proof}

For \eqref{eq:JK-endpoint}, we apply \cite[Lemma~2.3]{JerisonKenig1985} with
\(t=m-\frac12\).  The constant in \cite[Lemma~2.3]{JerisonKenig1985} depends, apart from \(n\),
only on the admissibility parameter
\[
 \operatorname{dist}(t_m,\mathbb Z)=\frac12,
\]
and is therefore independent of \(m\).

For \eqref{eq:Stein-endpoint}, we apply \cite[Lemma~4]{Stein1985}.  The relevant admissibility parameter is
\[
 \Delta(t_m)
 =\max\{m-t_m,t_m-(m-1)\}
 =\frac12
\]
and the hypotheses of \cite[Lemma~4]{Stein1985} reduce to
\[
 n-1<\beta=n-\frac34<n-\Delta(t_m)=n-\frac12,
 \qquad
 \frac1R=\frac1S-\frac{\beta}{n},
 \qquad
 1<S<\frac n\beta.
\]
which follow from \eqref{eq:SR-relations}.  Since \(\Delta(t_m)\)
is fixed, we get \eqref{eq:Stein-endpoint} with constants
independent of \(m\).
\end{proof}
\subsection{Analyticity and Antilocality}
We now recall certain properties of the fractional Laplace. We use the natural tail space
\[
    L_s^1(\R^n)
    :=\left\{f:\int_{\R^n}\frac{|f(x)|}{1+|x|^{n+2s}}\,dx<\infty\right\}.
\]
\begin{lemma}\label{lem:analytic}
Let $h\in L_s^1(\R^n)$ satisfy $(-\Delta)^s h=0$ distributionally in $B_{2R}$. Then there exist $\rho_0\in(0,R)$ and $H>0$ such that, for every $m\ge1$,
\begin{equation}\label{eq:h-expansion}
    h=T_{m-1}h+h_m\qquad\text{in }B_{\rho_0},
\end{equation}
where $T_{m-1}h$ is the Taylor polynomial of $h$ at the origin of degree at
most $m-1$, and
with
\begin{equation}\label{eq:analytic-rem}
    |h_m(x)|\le H\rho_0^{-m}|x|^m,
    \qquad |x|<\rho_0.
\end{equation}
\end{lemma}

\begin{proof}
From \cite[Main Theorem]{CarbottiEtAl2024}, \(h\) is real analytic in \(B_{2R}\).  Hence
there are \(\rho_0\in(0,R)\) and \(H>0\) such that the Taylor series
\[
    h(x)=\sum_{\gamma\in\mathbb N_0^n}a_\gamma x^\gamma
\]
converges absolutely for \(|x|<\rho_0\) and
\[
    \sum_{\gamma\in\mathbb N_0^n}|a_\gamma|\rho_0^{|\gamma|}\le H.
\]
Consequently, for \(|x|<\rho_0\),
\[
\begin{aligned}
 |h(x)-T_{m-1}h(x)|
 &\le \sum_{|\gamma|\ge m}|a_\gamma||x|^{|\gamma|}  \\
 &\le \left(\frac{|x|}{\rho_0}\right)^m
       \sum_{\gamma\in\mathbb N_0^n}|a_\gamma|\rho_0^{|\gamma|}
 \le H\rho_0^{-m}|x|^m.
\end{aligned}
\]
\end{proof}

\begin{lemma}[Antilocality]\label{lem:antilocality}
Let $u\in H^s(\mathbb R^n)$. If $u=0$ and
$(-\Delta)^s u=0$ distributionally on a nonempty open set, then
$u\equiv0$ in $\mathbb R^n$.
\end{lemma}
\begin{proof}
This is the antilocality property of the fractional Laplacian; see
\cite{GhoshSaloUhlmann2020,BergerSchilling2026,Prasad2026}.
\end{proof}

\section{$L^Q \rightarrow L^P$ Carleman Estimates}\label{sec:aux}
In this section, we prove Carleman estimates needed in the proof of Theorem~\ref{thm:intro-main}.

\begin{lemma}\label{lem:low-order-interpolation}
For $P,Q$ defined by \eqref{eq:PQ} we have,
\begin{equation}\label{eq:Talpha-PQ}
    \|T_\alpha^{(m)}g\|_{L^P(d\mu)}
    \le C_{n,\alpha}
    \|g\|_{L^Q(d\mu)}
\end{equation}
for every $m\ge1$, with $C_{n,\alpha}$ independent of $m$.
\end{lemma}

\begin{proof}
We apply Stein's interpolation theorem for analytic families
\cite[Theorem~1]{Stein1956} to
\[
    w\longmapsto T_{\beta w}^{(m)}
\]
on the strip \(0\le\Re w\le1\).  The boundary estimates are
\eqref{eq:JK-endpoint} and \eqref{eq:Stein-endpoint}; their constants have
the required exponential growth in \(\Im w\), uniformly in \(m\).  At
\(\theta=\alpha/\beta\), the interpolated exponents are
\[
    \frac1Q=\frac{1-\theta}{2}+\frac{\theta}{S},
    \qquad
    \frac1P=\frac{1-\theta}{2}+\frac{\theta}{R}.
\]
Using \eqref{eq:SR} gives exactly \eqref{eq:PQ}.  Since both boundary
constants are independent of \(m\), so is the interpolated constant.
\end{proof}

For compactly supported $F$ define, initially when the integrals converge absolutely,
\begin{equation}\label{eq:PmF}
 P_{m-1}^F(x)
 :=c_{n,\alpha}\sum_{j=0}^{m-1}\frac1{j!}
 \int_{\mathbb R^n}
 \left.\partial_r^j|rx-y|^{\alpha-n}\right|_{r=0}
 F(y)\,dy,
\end{equation}
and set
\begin{equation}\label{eq:RmF}
    R_mF:=I_\alpha F-P_{m-1}^F.
\end{equation}
For each fixed \(j\),
\[
 \left.\partial_r^j|rx-y|^{\alpha-n}\right|_{r=0}
 =
 (x\cdot\nabla)^j\bigl[|z-y|^{\alpha-n}\bigr]_{z=0}.
\]
Hence it is a homogeneous polynomial of degree \(j\) in \(x\). In particular, $P_{m-1}^F$ is a polynomial of degree at most $m-1$ whenever its coefficients are well defined.

\begin{proposition}\label{prop:remainder}
There exists a constant $C_0=C_0(n,\alpha) >0 $ such that, for every $m\ge1$ and every compactly supported $F$ satisfying
\[
    |x|^{-a_m}F\in L^Q(\mathbb R^n),
\]
the polynomial $P_{m-1}^F$ and the remainder $R_mF$ are well defined and
\begin{equation}\label{eq:remainder-est}
    \||x|^{-a_m}R_mF\|_{L^P(\mathbb R^n)}
    \le
    C_0\||x|^{-a_m}F\|_{L^Q(\mathbb R^n)}.
\end{equation}
In particular, $C_0$ is independent of $m$.
\end{proposition}

\begin{proof}
We first consider
\[
    F\in C_c^\infty(\mathbb R^n\setminus\{0\})
\]
and set
\[
    g(y):=|y|^{-t_m+\alpha}F(y).
\]
Substitution in \eqref{eq:Kz-cited} at \(z=\alpha\), using
\(d\mu(y)=|y|^{-n}\,dy\), gives
\[
    T_\alpha^{(m)}g(x)
    =
    \frac1{\Gamma((n-\alpha)/2)}
    |x|^{-t_m}R_mF(x).
\]
Hence Lemma~\ref{lem:low-order-interpolation} yields
\[
 \||x|^{-t_m-n/P}R_mF\|_{L^P}
 \le C_{n,\alpha}
 \||x|^{-t_m+\alpha-n/Q}F\|_{L^Q}.
\]
Since \(\alpha-n/Q=-n/P\) and
\(a_m=t_m+n/P\), this is \eqref{eq:remainder-est} for smooth $F$
supported away from the origin.

It remains to extend the definition of \(P_{m-1}^F\) and the estimate to a
compactly supported \(F\) such that
\[
    G:=|x|^{-a_m}F\in L^Q(\mathbb R^n).
\]
Let \(K\) be a compact set containing \(\supp F\).  For
\(0\le j\le m-1\),
\[
 \left|
 \left.\partial_r^j|rx-y|^{\alpha-n}\right|_{r=0}
 \right|
 \le C_j|x|^j|y|^{\alpha-n-j}.
\]
Thus each coefficient of \(P_{m-1}^F\) is bounded by a constant times
\[
 \int_K |y|^{a_m+\alpha-n-j}|G(y)|\,dy.
\]
The most singular case is \(j=m-1\).  By
\eqref{eq:critical-relations},
\[
    a_m+\alpha-n-(m-1)
    =
    \frac12-\frac n{Q'}.
\]
Therefore
\[
 \int_K |y|^{Q'(a_m+\alpha-n-(m-1))}\,dy
 \lesssim
 \int_0^1 r^{Q'/2-1}\,dr<\infty.
\]
H\"older's inequality shows that every coefficient functional defining
\(P_{m-1}^F\) extends continuously to
\(L^Q(|x|^{-a_mQ}dx)\).  In particular, \(P_{m-1}^F\) is well defined.

We now choose
\[
    F_\nu\in C_c^\infty(\mathbb R^n\setminus\{0\}),
    \qquad \supp F_\nu\subset K',
\]
for one fixed compact set \(K'\), such that
\[
    F_\nu\to F
    \quad\text{in }L^Q(|x|^{-a_mQ}dx).
\]
Since \(K'\) is bounded, this also implies \(F_\nu\to F\) in \(L^Q\).
The Hardy--Littlewood--Sobolev inequality therefore gives
\[
    I_\alpha F_\nu\to I_\alpha F
    \qquad\text{in }L^P.
\]
By the continuity of the coefficient functionals,
\[
    P_{m-1}^{F_\nu}\to P_{m-1}^F
\]
coefficient-wise, and hence in \(L^P\) on every bounded annulus.

On the other hand, the smooth estimate applied to \(F_\nu-F_\mu\) shows
that \(R_mF_\nu\) is Cauchy in
\(L^P(|x|^{-a_mP}dx)\).  Let \(\widetilde R_mF\) denote its limit.  On every
annulus \(A=\{r<|x|<R\}\), the weight \(|x|^{-a_m}\) is comparable to a
constant.  Thus
\[
 R_mF_\nu
 =
 I_\alpha F_\nu-P_{m-1}^{F_\nu}
 \longrightarrow
 I_\alpha F-P_{m-1}^F
 \quad\text{in }L^P(A),
\]
while the weighted convergence gives
\(R_mF_\nu\to\widetilde R_mF\) in \(L^P(A)\).  Hence
\[
    \widetilde R_mF=I_\alpha F-P_{m-1}^F
\]
almost everywhere on every annulus, and thus almost everywhere in
\(\mathbb R^n\).  This identifies the weighted limit with \(R_mF\).
Passing to the limit in the smooth estimate proves
\eqref{eq:remainder-est}.
\end{proof}

\section{Flatness and the Cancellation Lemma}\label{sec:cancellation}

We now use the flatness hypothesis \eqref{eq:flatness} in conjunction with  Carleman estimates to prove a cancellation lemma which leaves us with a representation of $u$ to which the
Carleman estimate can be applied once again in the proof of the main
theorem.

\begin{lemma}\label{lem:weighted-flatness}
Suppose $u\in H^{\alpha/2}(\R^n)$ and
\[
    \int_{B_r}|u|^2\,dx=O(r^N)
\]
for every $N>0$. Then, for every $B>0$,
\[
    |x|^{-B}u\in L^2(B_1),
\]
and, for every $A>0$,
\begin{equation}\label{eq:weighted-P}
    |x|^{-A}u\in L^P(B_1).
\end{equation}
\end{lemma}

\begin{proof}
On $A_j=B_{2^{-j}}\setminus B_{2^{-j-1}}$,
\[
 \int_{A_j}|x|^{-2B}|u|^2\,dx
 \le C2^{2Bj}\int_{B_{2^{-j}}}|u|^2\,dx
 \le C_N2^{(2B-N)j}.
\]
Choose $N>2B$ and sum in $j$. The fractional Sobolev embedding
$H^{\alpha/2}(\R^n)\hookrightarrow L^{p_*}(\R^n)$, $p_*=2n/(n-\alpha)$ (see,
e.g., \cite[Theorem~6.5]{DiNezzaPalatucciValdinoci2012}), gives $u\in
L^{p_*}(\R^n)$. Since $2<P<p_*$, choose $\theta\in(0,1)$ with
\[
    \frac1P=\frac\theta2+\frac{1-\theta}{p_*}.
\]
Then
\[
 \||x|^{-A}u\|_{L^P(B_1)}
 \le
 \||x|^{-A/\theta}u\|_{L^2(B_1)}^\theta
 \|u\|_{L^{p_*}(B_1)}^{1-\theta}<\infty.
\]
\end{proof}

Fix \(R>0\) such that \(B_{4R}\Subset\Omega\), choose
\[
    \chi\in C_c^\infty(B_{3R}),
    \qquad
    \chi\equiv1\quad\text{on }B_{2R},
\]
and set
\[
    F:=\chi Vu,
    \qquad
    h:=u-I_\alpha F.
\]

\begin{lemma}\label{lem:local-decomp}
For every \(A>0\),
\begin{equation}\label{eq:F-all-weights}
    |x|^{-A}F\in L^Q(\mathbb R^n).
\end{equation}
Moreover,
\[
    h\in L_s^1(\mathbb R^n),
\]
and
\begin{equation}\label{eq:h-harmonic}
    (-\Delta)^s h=0
    \qquad\text{in }B_{2R}.
\end{equation}
\end{lemma}

\begin{proof}
Fix \(A>0\). Since \(\chi\equiv1\) on \(B_R\), Lemma~\ref{lem:weighted-flatness}
and \eqref{eq:critical-holder} give
\[
 \||x|^{-A}F\|_{L^Q(B_R)}
 \le
 \|V\|_{L^{n/\alpha}(B_R)}
 \||x|^{-A}u\|_{L^P(B_R)}
 <\infty.
\]
On \(B_{3R}\setminus B_R\), one has \(|x|^{-A}\le R^{-A}\). Hence
\[
 \||x|^{-A}F\|_{L^Q(B_{3R}\setminus B_R)}
 \le
 C_{A,R}
 \|V\|_{L^{n/\alpha}(B_{3R})}
 \|u\|_{L^P(B_{3R})}
 <\infty.
\]
Since \(F\) is supported in \(B_{3R}\), this proves
\eqref{eq:F-all-weights}. In particular, \(F\in L^Q(\mathbb R^n)\).
Hence, by \eqref{eq:HLS},
\[
    I_\alpha F\in L^P(\mathbb R^n).
\]
The identity \eqref{eq:Riesz-inverse} gives
\[
    (-\Delta)^s I_\alpha F=F
\]
distributionally. Since \(\chi\equiv1\) on \(B_{2R}\),
\[
    F=Vu=(-\Delta)^s u
    \qquad\text{in }B_{2R},
\]
which proves \eqref{eq:h-harmonic}.

Finally,
\[
    u\in L^2(\mathbb R^n)\subset L_s^1(\mathbb R^n),
    \qquad
    I_\alpha F\in L^P(\mathbb R^n)\subset L_s^1(\mathbb R^n),
\]
and therefore \(h=u-I_\alpha F\in L_s^1(\mathbb R^n)\).
\end{proof}

Using Lemma~\ref{lem:analytic} we now find \(\rho_0\in(0,R)\) and \(H>0\)
such that, for every \(m\ge1\),
\[
    h=T_{m-1}h+h_m
    \quad\text{in }B_{\rho_0},
    \qquad
    |h_m(x)|\le H\rho_0^{-m}|x|^m
    \quad (|x|<\rho_0).
\]

\begin{lemma}[Cancellation Lemma]\label{lem:jet}
With \(F,h,\rho_0,H\) as above, for every \(m\ge1\),
\begin{equation}\label{eq:key-representation}
    u=R_mF+h_m
    \qquad\text{in }B_{\rho_0},
\end{equation}
where \(h_m\) is the analytic remainder in \eqref{eq:h-expansion} and
satisfies \eqref{eq:analytic-rem}.
\end{lemma}

\begin{proof}
By \eqref{eq:F-all-weights} and Proposition~\ref{prop:remainder},
\[
    I_\alpha F=P_{m-1}^F+R_mF.
\]
By \eqref{eq:h-expansion},
\[
    h=T_{m-1}h+h_m
    \qquad\text{in }B_{\rho_0}.
\]
Since \(u=I_\alpha F+h\), we obtain
\begin{equation}\label{eq:decomp-Q}
    u
    =
    Q_{m-1}+R_mF+h_m
    \qquad\text{in }B_{\rho_0},
\end{equation}
where
\[
    Q_{m-1}:=P_{m-1}^F+T_{m-1}h.
\]
Both summands are polynomials of degree at most \(m-1\), and hence so is
\(Q_{m-1}\).

We claim that \(Q_{m-1}\equiv0\). Suppose otherwise, and let
\(k\le m-1\) be the lowest degree of a nonzero homogeneous component of
\(Q_{m-1}\). Then there are \(c_Q>0\) and \(\rho_1\in(0,\rho_0)\) such that
\begin{equation}\label{eq:poly-lower}
    \|Q_{m-1}\|_{L^2(B_\rho)}
    \ge c_Q\rho^{k+n/2}
    \qquad (0<\rho<\rho_1).
\end{equation}

Since \(P>2\), choose \(\ell\in(1,\infty)\) such that
\[
    \frac1\ell=\frac12-\frac1P.
\]
Using Proposition~\ref{prop:remainder}, for \(0<\rho<\rho_1\),
\[
\begin{aligned}
 \|R_mF\|_{L^2(B_\rho)}
 &\le
 \||x|^{-a_m}R_mF\|_{L^P(\mathbb R^n)}
 \||x|^{a_m}\|_{L^\ell(B_\rho)}\\
 &\le
 C_{\mathrm{rem}}
 \||x|^{-a_m}F\|_{L^Q(\mathbb R^n)}
 \||x|^{a_m}\|_{L^\ell(B_\rho)}\\
 &\le C_m\rho^{a_m+n/\ell}
 =
 C_m\rho^{m-\frac12+\frac n2}.
\end{aligned}
\]
Here we used
\[
    a_m+\frac n\ell
    =
    \left(m-\frac12+\frac nP\right)
    +\left(\frac n2-\frac nP\right)
    =
    m-\frac12+\frac n2.
\]
Moreover, \eqref{eq:analytic-rem} gives
\[
    \|h_m\|_{L^2(B_\rho)}
    \le C_m'\rho^{m+n/2}.
\]
Since \(k\le m-1\), both exponents
\[
    m-\frac12+\frac n2
    \qquad\text{and}\qquad
    m+\frac n2
\]
are strictly larger than \(k+n/2\). Infinite-order vanishing also gives
\[
    \|u\|_{L^2(B_\rho)}
    =
    o(\rho^{k+n/2})
    \qquad\text{as }\rho\downarrow0.
\]
Therefore, by \eqref{eq:decomp-Q},
\[
 \|Q_{m-1}\|_{L^2(B_\rho)}
 \le
 \|u\|_{L^2(B_\rho)}
 +\|R_mF\|_{L^2(B_\rho)}
 +\|h_m\|_{L^2(B_\rho)}
 =
 o(\rho^{k+n/2}),
\]
which contradicts \eqref{eq:poly-lower}. Hence \(Q_{m-1}\equiv0\), and
\eqref{eq:key-representation} follows.
\end{proof}

\section{Proof of the Main Theorem}\label{sec:mainproof}

\begin{proof}[Proof of Theorem~\ref{thm:intro-main}]
By a translation, we may assume without loss of generality that $x_0=0$. We choose
\[
    0<\rho<\min\{\rho_0,1\}
\]
so small that
\[
    B_\rho\subset B_{2R}
\]
and
\begin{equation}\label{eq:critical-smallness}
    C_0
    \|V\|_{L^{n/\alpha}(B_\rho)}
    \le\frac12,
\end{equation}
where $C_0$ is the constant in
Proposition~\ref{prop:remainder}.

For each $m\ge1$ we set
\[
    X_m
    :=
    \||x|^{-a_m}u\|_{L^P(B_\rho)}.
\]
which is finite due to
Lemma~\ref{lem:weighted-flatness}. From
\eqref{eq:key-representation},
Proposition~\ref{prop:remainder}, and
\eqref{eq:analytic-rem},
\begin{equation}\label{eq:X-start}
 X_m
 \le
 C_0
 \||x|^{-a_m}F\|_{L^Q(\R^n)}
 +
 C_HH\rho_0^{-m}\rho^{1/2},
\end{equation}
where the second term follows from
\[
    m-a_m=\frac12-\frac nP
\]
and
\[
 \||x|^{-a_m}h_m\|_{L^P(B_\rho)}^P
 \le
 H^P\rho_0^{-mP}\omega_n
 \int_0^\rho r^{P/2-1}\,dr.
\]

Now on $B_\rho$,
\[
    F=Vu,
\]
so by \eqref{eq:critical-holder} we have
\begin{equation}\label{eq:critical-product-proof}
 \||x|^{-a_m}Vu\|_{L^Q(B_\rho)}
 \le
 \|V\|_{L^{n/\alpha}(B_\rho)}
 X_m,
\end{equation}
while on $\R^n\setminus B_\rho$ we get
\begin{equation}\label{eq:outer}
 \||x|^{-a_m}F\|_{L^Q(\R^n\setminus B_\rho)}
 \le
 \rho^{-a_m}
 \|F\|_{L^Q}.
\end{equation}
Combining \eqref{eq:X-start}--\eqref{eq:outer} and using
\eqref{eq:critical-smallness} we get
\begin{equation}\label{eq:X-absorbed}
 X_m
 \le
 2C_0
 \rho^{-a_m}
 \|F\|_{L^Q}
 +
 2C_HH
 \rho_0^{-m}
 \rho^{1/2}.
\end{equation}

On $B_{\rho/2}$ we have
\[
    |x|^{a_m}
    \le
    (\rho/2)^{a_m}
\]
and since
\[
    a_m=m-\frac12+\frac nP,
\]
\eqref{eq:X-absorbed} yields
\begin{equation}\label{eq:geometric-final}
 \|u\|_{L^P(B_{\rho/2})}
 \le
 C_1
 2^{-m}
 \|F\|_{L^Q}
 +
 C_2H
 \rho^{n/P}
 \left(\frac{\rho}{2\rho_0}\right)^m,
\end{equation}
where $C_1,C_2$ are independent of $m$. We now let $m\to\infty$ to get
\[
    u=0
    \qquad\text{a.e. in }B_{\rho/2}.
\]
The equation implies
\[
    (-\Delta)^su=0
\]
in the same ball. Lemma~\ref{lem:antilocality} therefore yields
\[
    u\equiv0
    \qquad\text{in }\R^n.
\]
\end{proof}

\end{document}